\documentclass[letterpaper,oneside]{amsart}

\usepackage{amssymb,amsmath,latexsym,amsthm}%or any usual package
\usepackage[english]{babel}
\usepackage[hidelinks]{hyperref}
\usepackage{thmtools}
\usepackage{thm-restate}

\declaretheorem[name=Theorem,numberwithin=section]{thm}

\newtheorem{lemma}[thm]{Lemma}
\newtheorem{corollary}[thm]{Corollary}
\newtheorem{proposition}[thm]{Proposition}

\theoremstyle{definition}
\newtheorem{remark}[thm]{Remark}
\newtheorem{definition}[thm]{Definition}
\newtheorem{example}[thm]{Example}

\numberwithin{equation}{section}

\DeclareMathOperator{\Tr}{Tr}
\DeclareMathOperator{\Gal}{Gal}

\newcommand{\Fp}{\mathbb{F}_p}

\newcommand{\Qp}{\mathbb{Q}_p}

\newcommand{\Zp}{\mathbb{Z}_p}

\DeclareMathOperator{\Hom}{Hom}
\DeclareMathOperator{\Res}{res}
\DeclareMathOperator{\Fr}{Fr}

\begin{document}

\title{Schmid's Formula for Higher Local Fields}

%first author
\author{Matthew Schmidt}
\address{Matthew Schmidt\\
Department of Mathematics\\
State University of New York\\
Buffalo, NY 14260, USA}
\email{mwschmid@buffalo.edu}
%\urladdr{http://www....}



\maketitle

%
%\begin{resume} 
%En cas de la théorie du corps de classes local, le symbole de Schmid-Witt encode les données intéressantes de la théorie de la ramification des $p$-extensions de $K$ et peut, par exemple, être utilisé pour calculer les groupes de ramifications supérieurs de telles extensions. En 1936, Schmid a découvert une formule explicite pour le symbole de Schmid-Witt des extensions Artin-Schreier des corps locaux. Ensuite, sa formule a été généralisée aux extensions Artin-Schreier-Witt, mais toujours elle s’agit d’un corps local. Pour cet article, nous faisons des généralités à propos de la formule de Schmid pour calculer le symbole Artin-Schreier-Witt-Parshin pour les extensions Artin-Schreier-Witt de corps locaux bidimensionnels de caractéristique positive.
%\end{resume}

\begin{abstract}
 In local class field theory, the Schmid-Witt symbol encodes interesting data about the ramification theory of $p$-extensi-ons of $K$ and can, for example, be used to compute the higher ramification groups of such extensions. In 1936, Schmid discovered an explicit formula for the Schmid-Witt symbol of Artin-Schreier extensions of local fields. Later, his formula was generalized to Artin-Schreier-Witt extensions, but still over a local field. In this paper we generalize Schmid's formula to compute the Artin-Schreier-Witt-Parshin symbol for Artin-Schreier-Witt extensions of two-dimensional local fields of positive characteristic. 
\end{abstract}

\bigskip
 
%%%%%%%%%%%%%%%%%%%%%%%%%%%%%%%%%%%%%%%%%%%%%%
% Introduction
%%%%%%%%%%%%%%%%%%%%%%%%%%%%%%%%%%%%%%%%%%%%%%
\section{Introduction}

Fix a prime $p>0$. Let $k$ be a finite field of characteristic $p$ and $K=k((T))$. Denote by $W(K)$ the ring of Witt vectors of $K$ and by $W_n(K)$ the ring of truncated Witt vectors of $K$ of length $n$. Let $\wp: W(K)\to W(K)$ denote the map defined by $\wp(x_0,x_1,\cdots)= (x_0^p,x_1^p,\cdots)-(x_0,x_1,\cdots)$. Using the reciprocity map $\omega_K: K^\times\to G_K=\Gal(K^{ab}/K)$ from local class field theory, we can define the Schmid-Witt symbol:

\begin{align}\label{Schmid_Witt}
[\cdot,\cdot)_n: W_n(K)/\wp W_n(K)\times K^\times/(K^\times)^{p^n} &\to W_n(\Fp),\\
	[x+\wp W_n(K),y\cdot(K^\times)^{p^n})_n &= \alpha^{\omega_K(y)}-\alpha,	\nonumber
\end{align}

where $\alpha\in W_n(K^{sep})$ is such that $\wp(\alpha)=x$. When $n=1$, this symbol was first studied by Schmid \cite{Schmid} who discovered the explicit formula
\begin{align*}
[x,y)_1=\Tr_{k/\Fp}(\Res(x\cdot d\log y)),
\end{align*}
where $\Res:K\to k$ is the residue map and $d\log$ is logarithmic differentiation. For $n>1$, Witt \cite{Witt} generalized Schmid's result to Artin-Schreier-Witt extensions over $K$:
\begin{align*}
[x,y)_n=\pi_n(\Tr_{W(k)/W(\Fp)}(\Res(g^{(n-1)}(X)\cdot d\log Y))),
\end{align*}
with $X$ and $Y$ being certain liftings of $x$ and $y$, $\pi_n: W(\Fp)\to W_n(\Fp)$ the natural projection map and $g^{(n)}(x)$ the $n$th ghost vector component of the Witt vector $X$. See Proposition~3.5 in \cite{Thomas} for a full, modern treatment.

Recently, Kosters and Wan simplified Witt's generalization of Schmid's formula producing a version much closer to Schmid's original and vitally avoiding the computation of ghost vectors:
\begin{align}\label{line:KW}
[x,y)_n=\pi_n(\Tr_{W(k)/W(\Fp)}(\Res(\tilde{x}\cdot d\log \tilde{y}))),
\end{align}
where $\tilde{x}$ and $\tilde{y}$ are again certain liftings of $x$ and $y$ (see Theorem 1.2 in \cite{KostersWanPub}). In this paper, we generalize (\ref{line:KW}) to two-dimensional local fields of positive characteristic.

More precisely, put $K=k((S))((T))$ and let $K_2^{top}(K)$ be the topological Milnor $K$-group of $K$ (\cite{Fesenko}, Definition 4.1). There are several approaches to generalize local class field theory, and thus the Schmid-Witt symbol, to higher local fields. Kato, in a series of papers (\cite{Kato1}, \cite{Kato2} and \cite{Kato3}), employs cohomological machinery, much like Tate did in the classical case, to build the reciprocity map. The high dimensional analogue to the Schmid-Witt symbol can then be defined as we did in the classical case (\ref{Schmid_Witt}). On the other hand, Parshin \cite{Parshin} uses an explicit non-degenerate pairing called the Artin-Schreier-Witt-Parshin pairing (Definition~\ref{parshins_formula}) :
\begin{align}\label{aswp_reciprocity}
	[\cdot , \cdot)_n: W_n(K)/\wp W_n(K)\times K_2^{top}(K)/(K_2^{top}(K))^{p^n} &\to W_n(\Fp),
\end{align}
and after applying Artin-Schreier-Witt theory, he is able to use this pairing to derive a map from $K_2^{top}(K)$ to the pro-$p$ part of the absolute Galois group of $K$. Pasting together similarly derived maps for the tame and unramified parts, Parshin obtains the full reciprocity map. In this paper, we adopt Parshin's explicit approach. 

After passing to projective limits, the symbol (\ref{aswp_reciprocity}) becomes:
\begin{align}\label{line:ASWPsymbol}
[\cdot , \cdot): W(K)/\wp W(K)\times \widehat{K_2^{top}(K)} \to W(\Fp),
\end{align}
where $\widehat{K_2^{top}(K)}$ is the $p$-adic completion of the topological Milnor $K$-group. The main theorem of this paper is the generalization of Kosters and Wan's formula (\ref{line:KW}) to the symbol (\ref{line:ASWPsymbol}). To construct this generalization, we apply lifting/residue maps that allow us to express the Artin-Schreier-Witt-Parshin symbol as the trace of a residue. Take $R$ to be a certain two sided power series ring in $S$ and $T$ and let $\tilde{K}=\Fr(W(k))((S))((T))$. Define a reduction $\ \bar{}:W(K)/\wp W(K)\to R$ and lift $\ \hat{}: \widehat{K_2^{top}(K)}\to \widehat{K_2^{top}(\tilde{K})}$. For full details see section~\ref{explicit_local_symbol}. Our primary result is then:
\begin{restatable}{theorem}{mainformula}\label{main_formula}
For $x\in W(K)/\wp W(K)$ and $y\in \widehat{K_2^{top}({K})}$,
	\[
		[x,y)=\Tr_{W(k)/W(\Fp)}(\Res(\bar{x}\cdot d\log(\hat{y}))).
	\]
\end{restatable}

After establishing some structure theorems, Proposition~\ref{basis} and  Corollary~\ref{K2top_structure},  we prove Theorem~\ref{main_formula} by showing our formula coincides with Parshin's on certain fundamental terms and extend via bilinearity and continuity.

 Again, the core advantage of Theorem~\ref{main_formula} is that we avoid ghost vectors entirely which greatly simplifies computations. As an application, we apply these simplifications to compute the upper ramification groups of $G_{p^\infty}$, the maximal abelian pro-$p$ extension of $K$, as defined by Lomadze and Hyodo in \cite{Lomadze} and \cite{Hyodo}.

Lexicographically order $\mathbb{Z}^2$. For two vectors $\vec{r}=(r_1, r_2), \vec{m}=(m_1,m_2)\in\mathbb{Z}_{\geq 0}^2$, let $\ell(\vec{r}, \vec{m})$ be the minimum integer $\ell$ such that $(p^\ell m_1, p^\ell m_2)\nless \vec{r}$.
For any $\vec{m}=(m_1,m_2)\in\mathbb{Z}^2$, write $\gcd\vec{m}=\gcd (m_1,m_2)$  and let $\mathbb{Z}^2_{>0}$ be the set of vectors in $\mathbb{Z}^2$  greater than $(0,0)$. For $\vec{i}\in\mathbb{Z}_{>0}^2$, let $G_{p^\infty}^{\vec{i}}$ denote the upper ramification group of $G_{p^\infty}$ (Definition~\ref{def:ramgrp}). Then: 

\begin{restatable}{theorem}{ramgroups}\label{ram_groups}
For $\vec{i}\in\mathbb{Z}_{>0}^2$, there is a group isomorphism:
	\[
		\phi_{\vec{i}}: G_{p^\infty}^{\vec{i}}\cong \left (p^{\ell(\vec{i}, \vec{m})}W(k)\right)_{\substack{\vec{m}\in\mathbb{Z}_{>0}^2\\ p\nmid \gcd \vec{m}}}\subseteq \prod_{\substack{\vec{m}\in\mathbb{Z}_{>0}^2\\ p\nmid \gcd \vec{m}}} W(k),
	\]
such that $\phi_{\vec{j}}=\phi_{\vec{i}}\vert_{G_{p^\infty}^{\vec{j}}}$ whenever $\vec{j}\leq \vec{i}$.
\end{restatable}

Theorem~\ref{ram_groups} generalizes Theorem 1.1 in \cite{Thomas}, which describes the ramification groups of pro-$p$ abelian extensions for a one-dimensional local field $K$. Thus, our work can be interpreted as a first step generalization of Thomas' work in \cite{Thomas} to the case where the residue field is not only not finite but not even perfect.  

This paper was written under the supervision of my advisor, Hui June Zhu. I thank her for her constant advice and guidance.  I am also grateful to Michiel Kosters and Daqing Wan for their comments and suggestions. 

%%%%%%%%%%%%%%%%%%%%%%%%%%%%%%%%%%%%%%%%%%%%%%
% Preliminaries
%%%%%%%%%%%%%%%%%%%%%%%%%%%%%%%%%%%%%%%%%%%%%%
\section{Preliminaries}

\subsection{Witt Vectors}

Let $R$ be a commutative ring with unity. Denote by $W(R)$  the ring of Witt vectors over $R$ and  $W_m(R)$ the ring of truncated Witt vectors. Let the $n$th ghost vector component of a Witt vector $x=(x_i)_{i\geq 0}\in W(R)$ be
	\[
		g^{(n)}(x)=\sum_{i=0}^n p^ix_i^{p^{n-i}},
	\]
and when $p$ is invertible in $R$, let the bijection $g:=(g^{(0)}, g^{(1)}, \cdots): W(R)\to R^{\mathbb{Z}_{\geq 0}}$ be the ghost vector map. 

For an $a\in R$, define the Teichm\"uller lift of $a$ in $W(R)$ to be $[a]=(a, 0, \cdots)$, and for a Witt vector $(x_0, x_1,\cdots)\in W(R)$, define the shifting map $V$ by mapping $(x_0, x_1, \cdots)$ to $(0, x_0, x_1, \cdots)$.

For a detailed exposition of Witt vectors see \S II.6 in \cite{Serre}. 

\subsection{Higher Local Fields}\label{subsection:HLF}

For any complete discrete valuation field $K$, denote by $\bar{K}$ the residue field of $K$. A complete discrete valuation field $K$ is said to be an $n$-dimensional local field if there is a sequence of complete discrete valuation fields $K_0, \cdots, K_n=K$ such that:
\begin{enumerate}
	\item $K_0$ is a finite field.
	\item $\bar{K_i}=K_{i-1}$ for $1\leq i\leq n$.
\end{enumerate}
The field $K=k((S))((T))$ of iterated Laurent series over the finite field $k$ is a two dimensional local field with tower of complete discrete valuation fields $k((S))((T))\supseteq k((S))\supseteq k$. If $F$ is any unramified extension of $\Qp$ with valuation $v_F$, let $F\{\{S\}\}$ be the set of doubly infinite power series:
\begin{align*}
	F\{\{S\}\} = \left \{ \sum_{i=-\infty}^{\infty}a_iS^i : a_i\in F, \inf_i v_F(a_i)>-\infty, a_i\to 0\textrm{ as } i\to -\infty\right \}.
\end{align*}
One can show that $F\{\{S\}\}$ is itself a complete discrete valuation field and can iteratively define $F\{\{S_1\}\}\cdots \{\{S_n\}\}$ (see Section~2 of \cite{Morrow} for details).
For example, if $F=\Qp$, then the tower of complete discrete valuation fields of $\Qp\{\{S\}\}\{\{T\}\}$ is given by:
\[
	\Qp\{\{S\}\}\{\{T\}\}\supseteq \Fp((S))((T))\supseteq \Fp((S)) \supseteq \Fp.
\]
While we will utilize this higher local field later on, our focus will be on the positive characteristic case. 

For the rest of this paper let $K=k((S))((T))$ be the two dimensional local field of positive characteristic. 
Place an order on $\mathbb{Z}^2$ as follows: for $\vec{i}=(i_1,i_2), \vec{j}=(j_1,j_2)\in\mathbb{Z}^2$, 
\begin{align}\label{ordering}
		\vec{i}\leq\vec{j}\iff 
		\begin{cases}
  		    i_1\leq j_1, i_2=j_2\textrm{ or } \\
   		    i_2\leq j_2.
		\end{cases}
\end{align}
We can give the field $K$ a rank two valuation $v_K:K\to\mathbb{Z}^2$ with $\mathbb{Z}^2$ ordered as above. With respect to this valuation, $K$ has valuation ring and  maximal ideal:
	\begin{align*}
		O_K &= \{x\in K\ |\ v_K(x)\geq (0,0)\}\\
		m_K &= \{x\in K\ |\ v_K(x)> (0,0)\}, 
	\end{align*}
so that $O_K/m_K \cong k$. For a detailed construction of the valuation $v_K$,  see p.7 of \cite{Zhukov_HDLF}, and for a more thorough treatment of rings of integers of higher local fields, see Section~3 of \cite{Morrow}.

We will close this section with a couple structure theorems regarding $K$ and $W(K)$. Fix an $\alpha\in k$ with $\Tr_{k/\Fp}(\alpha)\neq 0$ so that $\{\alpha\}$ is an $\Fp$-basis for $k/\wp k$. Let
 $\mathcal{C}$ be an $\Fp$-basis of $k$ and let $\beta=[\alpha]\in W(k)$. Define the indexing set:
\begin{align*}
	\mathbb{I} &= \{\vec{i}\in\mathbb{Z}_{<0}^2\ |\ p\nmid\gcd\vec{i} \}.
\end{align*}

\begin{proposition}\label{basis}
The set 
	\[
		\mathcal{D}=\{\alpha\}\cup \{ cS^{i}T^{j} | c\in\mathcal{C}, (i,j)\in \mathbb{I}\}
	\]
is an $\Fp$-basis for $K/\wp K$.
\end{proposition}
\begin{proof}
See Lemma 2, Section 1 of \cite{Parshin}.
\end{proof}

\begin{corollary}\label{structurethm}
 Every $x\in W(K)$ has a unique representative of the form
 	\[
 		c\beta + \sum_{\substack{(i,j)\in\mathbb{I}}} c_{ij}[S^i][T^j] \in W(K)/\wp W(K),
 	\]
 	where $c\in W(\Fp)$ and $c_{ij}\in W(k)$ with $c_{ij}\to 0$ as $i\to\infty$ or $i\to -\infty$ for every fixed $j$ and $\lim_{j\to -\infty}\max_i \| c_{ij}\|=0$, where $\|\cdot \|$ is the norm on $W(k)$. 
\end{corollary}
\begin{proof}
The unique representation follows from Proposition 3.10 in \cite{KostersWan} and Proposition~\ref{basis}.
\end{proof}

\subsection{Milnor $K$-groups}

Denote the second Milnor $K$-group of $K$ by $K_2(K)$. We write $K_2(K)$ as a multiplicative abelian group on symbols $\{a,b\}$, with $a,b\in K^\times$.  Define the topological Milnor $K$-group, $K_2^{top}(K)$, to be the quotient of $K_2(K)$ by the intersection of all its neighborhoods of zero. (See \cite{Fesenko}.) Our use of the topological Milnor $K$-group is due to our reliance on Parshin's class field theory. 

Take $\vec{i}\in\mathbb{Z}^2$ with $\vec{i}>(0,0)$. Following Hyodo (\cite{Hyodo}, p.291), define the $\vec{i}$th unit group of $K_2^{top}(K)$ to be:
	\[
		U^{\vec{i}}K_2^{top}(K)=\{\{u,x\} | u,x\in K^\times, v_K(u-1)\geq \vec{i}\}.
	\]
Similarly, let $V_K=1+m_K\subset K$ and denote by $VK_2^{top}(K)$ the subgroup of $K_2^{top}(K)$ generated by elements of the form $\{u,x\}$,  $u\in V_K$ and $x\in K$. That is,
\[
	VK_2^{top}(K)=\{\{u,x\} | u, x\in K^\times, v_K(u-1)>(0,0)\},
\]
and we see that for any $\vec{i}>(0,0)$, $U^{\vec{i}}K_2^{top}(K)\subseteq VK_2^{top}(K)$.
Define two indexing sets:
\begin{align*}
	\mathbb{J}_S &= \{(i,j)\in\mathbb{Z}_{>0}^2\ |\ \gcd(i,p)\neq 1, \gcd(j,p)=1\}\\ 
	\mathbb{J}_T &= \{(i,j)\in\mathbb{Z}_{>0}^2\ |\ \gcd(i,p)=1 \},
\end{align*}
so that $VK_2^{top}(K)$ has the following decomposition:
\begin{proposition} \label{VK_decomp}

Every $y\in VK_2^{top}(K)$ can be written uniquely in the form:
	\[
		\prod_{\substack{(i,j)\in\mathbb{J}_S\\k\geq 0}} \{1+b_{ijk}S^iT^j, S\}^{d_{ijk}p^k}\prod_{\substack{(i,j)\in\mathbb{J}_T\\k\geq 0}} \{1+a_{ijk}S^iT^j, T\}^{c_{ijk}p^k},
	\]
for some $a_{ijk}, b_{ijk}\in k$ and $c_{ijk}, d_{ijk}\in [0,p-1]$.
\end{proposition}
\begin{proof}
See Proposition 2.4 in \cite{Fesenko2}.
\end{proof}

\begin{proposition}\label{VKK_comp}
There is a group isomorphism:
	\[
		VK_2^{top}(K)\cong \varprojlim_n VK_2^{top}(K)/(VK_2^{top}(K))^{p^n}.
	\]
\end{proposition}
\begin{proof}
By Remark 1 in Section 4 in \cite{Fesenko} (p.496), the natural map
	\[
		VK_2^{top}(K)\to\varprojlim_nVK_2^{top}(K)/(VK_2^{top}(K))^{p^n}
	\]
is surjective. Moreover, by an earlier remark following 4.2 in \cite{Fesenko} (p.493), 
	\[
		\cap_n(VK_2^{top}(K))^{p^n}=\{1\},
	\]
so the natural map must also be injective. 
\end{proof}

\begin{corollary}\label{K2top_structure}
Let $\widehat{K_2^{top}(K)}=\varprojlim_n K_2^{top}(K)/(K_2^{top}(K))^{p^n}$. Every $y\in \widehat{K_2^{top}(K)}$ can be written uniquely in the form:
\begin{align*}
	y=\{S^e, T\}\prod_{\substack{(i,j)\in\mathbb{J}_S\\k\geq 0}} \{1+b_{ijk}S^iT^j, S\}^{d_{ijk}p^k}\prod_{\substack{(i,j)\in\mathbb{J}_T\\k\geq 0}} \{1+a_{ijk}S^iT^j, T\}^{c_{ijk}p^k},
\end{align*}
	for some $e\in\Zp$, $a_{ijk}, b_{ijk}\in k$ and $c_{ijk}, d_{ijk}\in [0,p-1]$.
\end{corollary}
\begin{proof}

It is known that $K_2^{top}(K)\cong \langle\{S,T\}\rangle\times (k^\times)^2\times VK_2^{top}(K)$ (See  \cite{Fesenko}, p.493). So let $\{S,T\}^i\in \langle\{S,T\}\rangle$. Then for any $i>0$, $\{S,T\}^i=\{S^i,T\}$, and we see the map $\langle\{S,T\}\rangle\to \mathbb{Z}$ given by $\{S^i,T\}\mapsto i$ is an isomorphism. This implies $\langle\{S,T\}\rangle/\langle\{S,T\}\rangle^{p^n}\cong \mathbb{Z}/p^n\mathbb{Z}$.

Because $k^p=k$, $ k^\times/(k^\times)^{p^n}=1$ and $\varprojlim_n k^\times/(k^\times)^{p^n}=1$. Then, by Proposition~\ref{VKK_comp}, we see that:
\begin{align*}
\widehat{K_2^{top}(K)}&\cong\varprojlim_n \langle\{S,T\}\rangle/\langle\{S,T\}\rangle^{p^n}\times VK_2^{top}(K)/(VK_2^{top}(K))^{p^n}\\
	&\cong\Zp\times\varprojlim_n VK_2^{top}(K)/(VK_2^{top}(K))^{p^n}\\
	&\cong\Zp\times VK_2^{top}(K).
\end{align*}
The Corollary then follows from Proposition~\ref{VK_decomp}.
\end{proof}

%%%%%%%%%%%%%%%%%%%%%%%%%%%%%%%%%%%%%%%%%%%%%%
% The Explicit Local Symbol
%%%%%%%%%%%%%%%%%%%%%%%%%%%%%%%%%%%%%%%%%%%%%%
\section{The Explicit Local Symbol} \label{explicit_local_symbol}

Before proving Theorem~\ref{main_formula}, we assemble all the pieces required to define the Artin-Schreier-Witt-Parshin symbol.

\subsection{The Artin-Schreier-Witt-Parshin Symbol}

Let $O_F = W(k)$, $F=\Fr (O_F)$ and $\tilde{K}=F((S))((T))$. There is a natural lifting from $K$ to $\tilde{K}$ which induces liftings on both $W(K)$ and $\widehat{K_2^{top}(K)}$ in the following way.
We will write both liftings by $\hat{\ }$.

\begin{definition}
One can lift $W(K)$ to $W(\tilde{K})$ by the map:
\begin{align*}
	x = (x_i)_{i\geq 0} &\mapsto \hat{x}=(\hat{x}_i)_{i\geq 0},
\end{align*}
where $\hat{x}_i\in O_F((S))((T))$ is any lifting of $x_i$ under the canonical projection \\$O_F((S))((T))\to K$.
\end{definition}

\begin{definition} 
By Corollary~\ref{K2top_structure}, every $y\in \widehat{K_2^{top}(K)}$ can be written uniquely in the form:
\begin{align*}
y=\{S^e, T\}\prod\{1+a_{ij}S^iT^j, S\}\cdot\prod\{1+b_{ij}S^iT^j,T\}.
\end{align*}
We denote by $\hat{y}$ the lifting of $y$ to $\widehat{K_2^{top}(\tilde{K})}$ given by
\begin{align*}
	\hat{y}=\{S^e, T\}\prod\{1+[a_{ij}]S^iT^j, S\}\cdot\prod\{1+[b_{ij}]S^iT^j,T\}.
\end{align*}
\end{definition}

Parshin's map also makes use of the residue and logarithmic derivative maps:
\begin{definition}\label{residue}
Let $\omega=\sum_{i,j}a_{ij}S^iT^jdS\wedge dT\in \Omega^2_{\tilde{K}/F}$ where $a_{ij}\in F$. Define the map:
\begin{align*}
	\Res:\Omega^2_{\tilde{K}/F} &\to F\\
	\Res(\omega)&= a_{-1,-1}.
\end{align*}
\end{definition}

\begin{definition}\label{dlog}
For  $\{a_1,a_2\}\in K_2(\tilde{K})$ let
\begin{align*}
	d\log: K_2(\tilde{K})&\to \Omega^2_{\tilde{K}/F}\\ 
	d\log(\{a_1,a_2\})&= \frac{da_1}{a_1}\wedge \frac{da_2}{a_2},
\end{align*}
and note that this map is well-defined by the argument on p.166 of \cite{Parshin}.
\end{definition}

It is now possible to define the higher dimensional Schmid-Witt symbol  (See Section 3 of \cite{Parshin}):
\begin{definition}\label{parshins_formula} (Parshin's Formula)
The Artin-Schreier-Witt-Parshin symbol is given by:
\begin{align*}
[\cdot , \cdot)_n:&W_n(K)/\wp W_n(K)\times K_2^{top}(K)/(K_2^{top}(K))^{p^n}\to W_n(\Fp)=\mathbb{Z}/p^n\mathbb{Z}\\
&[x, y)_n= \Tr_{W_n(k)/W_n(\Fp)}(g^{-1}((\Res(g^{(i)}(\hat{x})\cdot d\log \hat{y}))_{i=0}^{n-1})\bmod p),
\end{align*}
where the $\bmod\ p$ is taken component-wise. Note that $\Res(g^{(i)}(\hat{x})\cdot d\log \hat{y})$ lies in the subring $W(k)\subset F$ so the $\bmod\ p$ makes sense.
\end{definition}

\subsection{Generalized Schmid's Formula}

In this section we will prove Theorem~\ref{main_formula}.
Let $R = F\{\{S\}\}\{\{T\}\}$ so that $R$ is the 3-dimensional local field of characteristic zero with residue field $K$. Note that $R$ can be equipped with an analogous mapping $\Res: R\to W(k)$ by sending $\sum_{i,j}a_{ij}S^iT^j\mapsto a_{-1,-1}$. 

\begin{definition}
If $x\in W(K)/\wp W(K)$, there is a reduction map to $\bar{x}\in R$ by sending
\begin{align*}
x=c\beta + \sum_{\substack{(i,j)\in\mathbb{I}}} c_{ij}[S^i][T^j]\bmod \wp W(K)
\end{align*}
to:
\begin{align*}
	\bar{x}=c\beta + \sum_{\substack{(i,j)\in\mathbb{I}}} c_{ij}S^iT^j,
\end{align*}
where $c\in W(\Fp)$ and $c_{ij}\in W(k)$. (The existence and uniqueness of the representation follows from Corollary~\ref{structurethm}.)
\end{definition}

\begin{remark}
For notational convenience, define the rings
\begin{align*}
 O_{\tilde{K}}&=O_F((S))((T))\subset \tilde{K}\\
 O_R&=O_F\{\{S\}\}\{\{T\}\}\subset R. 
 \end{align*}
 For any $f\in O_{\tilde{K}}$ and $g\in O_R$, it's easy to see that $fg\in O_R$. By construction, $\hat{y}$ is generated by symbols of the form $\{1+aS^iT^j, S\}$ and $\{1+bS^iT^j,T\}$ with $a, b\in O_F$. Therefore, since $\bar{x}\in O_R$ and $d\log(\hat{y})\in \Omega^2_{O_{\tilde{K}}/O_F}$, we see $\bar{x}\cdot d\log(\hat{y})\in \Omega^2_{O_R/O_F}$ and so  $\Res(\bar{x}\cdot d\log(\hat{y}))\in O_F$. Thus, taking a trace from $O_F=W(k)$ to $W(\Fp)$ is well-defined and our formula in Theorem~\ref{main_formula} is consequently well-defined.
\end{remark}

Before proving the main theorem, we start with a lemma:
\begin{lemma}\label{ijell_lemma}
Suppose that $i,j,\ell_1,\ell_2\in\mathbb{Z}_{\geq 0}$ with $ij\neq 0$ such that $p\nmid\gcd(i,j)$ and $p\nmid\gcd(\ell_1,\ell_2)$. Then for all $h\in\mathbb{Z}_{\geq 0}$,
	$$\frac{\ell_1p^h}{i}=\frac{\ell_2p^h}{j}\in\mathbb{Z}_{\geq 0}\iff \frac{\ell_1}{i}=\frac{\ell_2}{j}\in\mathbb{Z}_{\geq 0}.$$
\end{lemma}
\begin{proof}
The reverse direction is trivial. So suppose $\frac{\ell_1p^h}{i}=\frac{\ell_2p^h}{j}\in\mathbb{Z}_{\geq 0}$ and $i\mid \ell_1p^h$ but $i\nmid \ell_1$. 
Taking the $p$-adic valuation of $\frac{\ell_1p^h}{i}=\frac{\ell_2p^h}{j}$ gives: 
\begin{align}\label{line:vp}
	v_p(\ell_1)+h-v_p(i)=v_p(\ell_2)+h-v_p(j).
\end{align}
Noting that $p\nmid j$, (\ref{line:vp}) becomes:
\begin{align}\label{line:vp_l1}
	v_p(\ell_1)-v_p(i)=v_p(\ell_2).
\end{align}
The condition $p\nmid \gcd(\ell_1,\ell_2)$ implies that $v_p(\ell_1)=0$ or $v_p(\ell_2)=0$. In the first case, $v_p(i)=-v_p(\ell_2)$, which cannot happen 
since $i\in\mathbb{Z}_{\geq 0}$. So $v_p(\ell_2)=0$. But then (\ref{line:vp_l1}) yields $v_p(\ell_1)=v_p(i)$, and because we assumed $i\mid \ell_1p^h$, we must have $i\mid \ell_1$, a contradiction. Finally, if $v_p(\ell_1)=0=v_p(\ell_2)$, then $v_p(i)=0$, also a contradiction.
\end{proof}

\mainformula*
\begin{proof}
Because we know the  structure of $W(K)/\wp W(K)$ and $\widehat{K_2^{top}({K})}$ via Proposition~\ref{basis} and  Corollary~\ref{K2top_structure} respectively,  and by the continuity and $\Zp$-bilinearity of $[\ ,\ )$ from Parshin (\cite{Parshin}, Proposition 6 and 7), it suffices to prove the claim for the following cases:
\begin{enumerate}
\item $[c[S^{\ell_1}T^{\ell_2}], \{S,T\})$, $c\in W(k)$, $\ell_1,\ell_2\in\mathbb{Z}$, $p\nmid\gcd(\ell_1,\ell_2)$.
\item $[b[S^{\ell_1}T^{\ell_2}], \{1+aS^iT^j, S\})$,  $[b[S^{\ell_1}T^{\ell_2}], \{1+aS^iT^j, T\})$, $a\in k$, $b\in W(k)$, $\ell_1,\ell_2, i, j\in\mathbb{Z}$, $p\nmid\gcd(\ell_1,\ell_2)$ and $p\nmid\gcd(i,j)$.
\end{enumerate}

\emph{Proof of (1):} 
Here $d\log(\{S,T\})=\frac{dS}{S}\wedge \frac{dT}{T}=S^{-1}T^{-1}dS\wedge dT$ and $\bar{x}=cS^{\ell_1}T^{\ell_2}$. Then,
\begin{align*}
\bar{x}\cdot d\log(\hat{y})&=cS^{\ell_1}T^{\ell_2}\cdot S^{-1}T^{-1}dS\wedge dT=cS^{\ell_1-1}T^{\ell_2-1} dS\wedge dT.
\end{align*}
Therefore if $\ell_1\neq 0$ or $\ell_2\neq 0$, $[c[S^{\ell_1}T^{\ell_2}], \{S, T\})=0$. If $\ell_1=\ell_2=0$, 
	\[
		\Tr_{W(k)/W(\Fp)}(\Res(cS^{\ell_1-1}T^{\ell_2-1} dS\wedge dT))=
		\Tr_{W(k)/W(\Fp)}(c).
	\]
On the other hand if $c=(c_k)_{k=0}^\infty$, then by Proposition 1.10 in \cite{Thomas_thesis},
	\[
		c[S^{\ell_1}T^{\ell_2}]=(c_kS^{\ell_1p^k}T^{\ell_2p^k})_{k=0}^\infty,
	\]
and so 
$\widehat{c[S^{\ell_1}T^{\ell_2}]}=
	([c_k]S^{\ell_1p^k}T^{\ell_2p^k})_{k=0}^\infty$.
Hence:
\begin{align*}
g^{(h)}({([c_k]S^{\ell_1p^k}T^{\ell_2p^k})_{k=0}^\infty})&=
	\sum_{k=0}^h p^k([c_k]S^{\ell_1p^k}T^{\ell_2p^k})^{p^{h-k}}\\
	&=(\sum_{k=0}^h p^k[c_k]^{p^{h-k}})S^{\ell_1p^h}T^{\ell_2p^h}.
\end{align*}
But then:
\begin{align}\label{Line2}
g^{(h)}(\hat{x})&\cdot S^{-1}T^{-1}dS\wedge dT=(\sum_{k=0}^h p^k[c_k]^{p^{h-k}})S^{\ell_1p^h-1}T^{\ell_2p^h-1}dS\wedge dT,
\end{align}
and consequently (\ref{Line2}) has nonzero residue if and only if $\ell_1=\ell_2=0$. If $\ell_1=\ell_2=0$, we get 
	\[
		g^{-1}((\sum_{k=0}^h p^k[c_k]^{p^{h-k}})_{h=0}^\infty)=g^{-1}((g^{h}(c'))_{h=0}^\infty)=c',
	\]
where $c'=([c_k])_{k=0}^\infty\in W(W(k))$.
Because $c'\equiv c\bmod p$, Parshin's formula finally gives:
\begin{align*}
 	\Tr_{W(k)/W(\Fp)}(g^{-1}((\Res(g^{(i)}\hat{x}&\cdot S^{-1}T^{-1}dS\wedge dT))_{i=0}^{m-1})\bmod p)\\&=\Tr_{W(k)/W(\Fp)}(c).
\end{align*}

\emph{Proof of (2):} For simplicity, we will assume $ij\neq 0$. The other cases are easier and can be proven similarly.
We first compute 
\begin{align*}
d\log(\{1+aS^iT^j, S\})&=\frac{d(1+aS^iT^j)}{1+aS^iT^j}\wedge\frac{dS}{S}
	\\&=-aj\sum_{k=0}^\infty(-a)^kS^{(k+1)i-1}T^{(k+1)j-1}dS\wedge dT.
\end{align*}
Likewise, $d\log(\{1+aS^iT^j, T\})=ai(1+aS^iT^j)^{-1}S^{i-1}T^{j-1}dS\wedge dT$. Then as before, $\bar{x}=cS^{\ell_1}T^{\ell_2}$, so 
\begin{align*}
\bar{x}\cdot d\log(\{1+aS^iT^j, S\})&
	=-ajc\sum_{k=0}^\infty(-a)^kS^{(k+1)i+\ell_1-1}T^{(k+1)j+\ell_2-1}.
\end{align*}
This has nonzero residue if and only if $k+1=-\frac{\ell_1}{i}=-\frac{\ell_2}{j}\in\mathbb{Z}_{\geq 1}$, that is, $k=\frac{-i-\ell_1}{i}=\frac{-j-\ell_2}{j}\in\mathbb{Z}_{\geq 0}$. (If $\ell_1=\ell_2=0$, then $(k+1)i-1=(k+1)j-1=-1$, so that $i=j=0$. The following implies that the resulting residue is zero regardless.) In this case, we get 
\begin{align}\label{Line3}
\Tr_{W(k)/W(\Fp)}(\Res(\bar{x}\cdot d\log(\{1+aS^iT^j, S\})))&=\Tr_{W(k)/W(\Fp)}(-jc(-a)^{\frac{\ell_1}{i}}).
\end{align}

Now, we compute the symbol using Parshin's formula. As before, \\ $g^{(h)}(\hat{x})=g^{(h)}(c')S^{\ell_1p^h}T^{\ell_2p^h}$. Then 
\begin{align*}
g^{(h)}(\hat{x})\cdot &d\log(\{1+aS^iT^j, S\})\\&=-ajg^{(h)}(c')\sum_{k=0}^\infty(-a)^kS^{(k+1)i+\ell_1p^h-1}T^{(k+1)j+\ell_2p^h-1}dS\wedge dT.
\end{align*}
This term has nonzero residue if and only if $k=-\frac{\ell_1p^h}{i}-1=-\frac{\ell_2p^h}{j}-1\in\mathbb{Z}_{\geq 0}$. (By Lemma~\ref{ijell_lemma}, this occurs exactly when $-\frac{\ell_1}{i}-1=-\frac{\ell_2}{j}-1\in\mathbb{Z}_{\geq 0}$.)  For this $k$, we get a residue $-ajg^{(h)}(c')(-a)^{-\frac{\ell_1p^h}{i}-1}=-jg^{(h)}(c')(-a)^{-\frac{\ell_1p^h}{i}}$.

Then, because $((-a)^{-\frac{\ell_1p^h}{i}})_{h=0}^\infty=(((-a)^{-\frac{\ell_1}{i}})^{p^h})_{h=0}^\infty=g([(-a)^{-\frac{\ell_1}{i}}])$, we see that:
\begin{align*}
	g^{-1}((-jg^{(h)}(c')(-a)^{-\frac{\ell_1p^h}{i}})_{h=0}^\infty)&=-jg^{-1}(g(c') g([(-a)^{-\frac{\ell_1}{i}}]))\\
		&=-jc'[(-a)^{-\frac{\ell_1}{i}}]\equiv -jc(-a)^{-\frac{\ell_1}{i}}\bmod p.
\end{align*}
Comparing this with (\ref{Line3}), the claim follows.
\end{proof}

\begin{remark}
Computations like those in the proof of Theorem~\ref{main_formula} can show that for $\ell, i\in\mathbb{Z}_{>0}$,
\begin{align*}
	[cS^\ell, \{1+aS^i, T\})=[cS^\ell, 1+aS^i) \\
	[cT^\ell, \{1+aT^i, S\})=[cT^\ell, 1+aT^i),
\end{align*}
where the symbol $[,)$ on the right hand side is the one-dimensional Schmid-Witt symbol from \cite{KostersWan} and \cite{Thomas} over a local field $K=k((S))$ and $K=k((T))$, respectively.
\end{remark}

Some further explicit computations using Theorem~\ref{main_formula} yield: 
\begin{corollary}\label{explicit_symbol}
 If $x=c\beta+\sum_{(i,j)\in\mathbb{I}}c_{ij}[S]^{-i}[T]^{-j}$ and 
 \[
 y=\{S^e, T\}\prod_{(i,j)\in\mathbb{J}_S}\{1+a_{ij}S^iT^j, S\}\cdot\prod_{(i,j)\in\mathbb{J}_T}\{1+b_{ij}S^iT^j,T\},
 \]
 then:
\begin{align*}
[x,y)=&ec\Tr_{W(k)/W(\Fp)}(\beta)\\
	&+\sum_{(m,n)\in\mathbb{I}} [ \sum_{\substack{(i,j)\in\mathbb{J}_S,\\\frac{m}{i}=\frac{n}{j}\in\mathbb{Z}_{\geq 0}}}(\sum_{k=0}^\infty c_kp^k\Tr_{W(k)/W(\Fp)}(-jc_{m,n}([-a_{ijk}]^{m/i})))\\
&+\sum_{\substack{(i,j)\in\mathbb{J}_T,\\\frac{m}{i}=\frac{n}{j}\in\mathbb{Z}_{\geq 0}}}(\sum_{k=0}^\infty d_kp^k\Tr_{W(k)/W(\Fp)}(ic_{m,n}([-b_{ijk}]^{m/i})))].
\end{align*}
\end{corollary}

%%%%%%%%%%%%%%%%%%%%%%%%%%%%%%%%%%%%%%%%%%%%%%
% Computation of $G^{\vec{i}}$
%%%%%%%%%%%%%%%%%%%%%%%%%%%%%%%%%%%%%%%%%%%%%%
\section{Computation of $G^{\vec{i}}$}\label{sec:compG}

Let $G_K=\Gal(K^{ab}/K)$ be the Galois group of the maximal abelian extension of $K$, $G_{p^n}=G_K/(G_K)^{p^n}$ be the Galois group of the maximal abelian extension of exponent $p^n$ of $K$, and $G_{p^\infty}=\varprojlim_n G_{p^n}$ be the Galois group of the maximal abelian pro-$p$ extension of $K$. For any $n$-dimensional local field, Hyodo \cite{Hyodo} and Lomadze \cite{Lomadze} studied analogues to the classical upper ramification groups. Here, we will apply Theorem~\ref{main_formula} to compute the ramification groups of $G_{p^\infty}$ in the two-dimensional case.
In order to define these groups, however, we need to introduce the reciprocity map.

In \cite{Parshin}, Parshin defines the reciprocity map by pasting together three compatible maps corresponding to the tame, unramified, and pro-$p$ extensions of $K$. Here, we sketch the formulation of the pro-$p$ part of the reciprocity map, which is the only part we will use. 
Let 
\[
	\mathfrak{M}(K)=\varinjlim W_n(K)/\wp W_n(K),
\]
 with the direct limit taken with respect to the shifting maps $V:W_n(K)\to W_{n+1}(K)$. Define the natural maps:
\begin{align*}
	\phi_n:\widehat{K_2^{top}(K)} &\to K_2^{top}(K)/(K_2^{top}(K))^{p^n}\\
	\psi_n: W_n(K)/\wp W_n(K) &\to \mathfrak{M}(K).
\end{align*}
Starting with the Artin-Schreier-Witt-Parshin symbol
\[
	[\cdot , \cdot)_n:W_n(K)/\wp W_n(K)\times K_2^{top}(K)/(K_2^{top}(K))^{p^n}\to W_n(\Fp)=\mathbb{Z}/p^n\mathbb{Z},
\]
Parshin applies a standard argument (\cite{KawadaSatake}, p.373) to yield the pairing
\begin{align*}
	P: \mathfrak{M}(K) \times \widehat{K_2^{top}(K)} &\to \mathbb{Q}/\mathbb{Z},\\
	P(x,y) &= [x_n, \phi_n(y)]_n,
\end{align*}
where $n$ is any integer such that $x_n\in W_n(K)/\wp W_n(K)$ with $x=\psi_n(x_n)$. From this symbol, we then get a map:
\begin{align}\label{omegak_init}
	\widehat{K_2^{top}(K)} &\to \Hom(\mathfrak{M}(K), \mathbb{Q}/\mathbb{Z})\\
	y &\mapsto (x\mapsto P(x, y)).\nonumber
\end{align}
However, because  $\mathfrak{M}(K)$ is the Pontryagin dual of $G_{p^\infty}$ via Artin-Schreier-Witt theory, the map on line (\ref{omegak_init}) yields a map $\omega_K$ from $ \widehat{K_2^{top}(K)} $ to $G_{p^\infty}$. This defines the pro-$p$ part of Parshin's reciprocity map. 
It can be further shown that this map is actually an isomorphism: 
\begin{proposition}\label{prop:Ginfiso}
For $n\geq 1$, $K_2^{top}(K)/(K_2^{top}(K))^{p^n}\cong G_{p^n}$ and therefore 
$$\widehat{K_2^{top}(K)}=\varprojlim_{n}K_2^{top}(K)/(K_2^{top}(K))^{p^n}\cong G_{p^\infty}.$$
\end{proposition}
\begin{proof}
By Artin-Schreier-Witt theory, $W_n(K)/\wp W_n(K)$ and $G_{p^n}$ are Pontryagin dual to each other. However, by Proposition 1 in \cite{Parshin2}, 
	\[
		K_2^{top}(K)/(K_2^{top}(K))^{p^n}\cong \Hom(W_n(K)/\wp W_n(K), \mathbb{Q}/\mathbb{Z}),
	\]
and so Pontryagin duality yields the isomorphism. The second isomorphism follows by taking the projective limit. 
\end{proof}

Now, with the isomorphism $\omega_K:\widehat{K_2^{top}(K)}\to G_{p^\infty}$, the ramification groups of $G_{p^\infty}$ can be defined:
\begin{definition}\label{def:ramgrp}
For $\vec{i}\in\mathbb{Z}^2$ with $\vec{i}>(0,0)$, define: 
	\[
		G_{p^\infty}^{\vec{i}}=\omega_K(U^{\vec{i}}K_2^{top}(K)).
	\]
\end{definition}

Like \cite{Thomas} and \cite{KostersWan}, we do not compute $G_{p^\infty}^{\vec{i}}$ directly, but instead compute its image in a decomposition of $H=\Hom(W(K)/\wp W(K), \Zp)$.
First, observe that by Artin-Schreier-Witt theory and a well known property of direct and projective limits:
\begin{align}\label{line:Ktopiso}
	 G_{p^\infty}&\cong \Hom(\mathfrak{M}(K), \mathbb{Q}/\mathbb{Z})= \Hom(\varinjlim_n W_n(K)/\wp W_n(K), \mathbb{Q}/\mathbb{Z}) \\ 
		&\cong \varprojlim_n \Hom(W_n(K)/\wp W_n(K), \mathbb{Z}/p^n\mathbb{Z})\cong H,\nonumber 
\end{align}
and so following the isomorphisms above, it is easy to see that $\omega_K$ has the equivalent form:
\begin{align}\label{line:parshinrec}
	\omega_K^H: \widehat{K_2^{top}(K)} &\to \Hom(W(K)/\wp W(K), \Zp) \\
	y &\mapsto (x\mapsto [x,y)).\nonumber
\end{align}
We can then consider a decomposition of $\Hom(W(K)/\wp W(K), \Zp)$ in which to compute $G_{p^\infty}^{\vec{i}}$:
\begin{lemma}\label{lemma:Hiso}
There is a group isomorphism 
\[
	\Hom(W(K)/\wp W(K), \Zp)\cong W(\Fp)\times \prod_{\substack{(i,j)\in\mathbb{J}_S}}W(k)\prod_{\substack{(i,j)\in\mathbb{J}_T}}W(k).
\]
\end{lemma}
\begin{proof}
For $\chi\in H$, observe that for any $x=c\beta+\sum_{(i,j)\in\mathbb{I}}c_{ij}[S]^{-i}[T]^{-j}\in W(K)/\wp W(K)$,
\[
	\chi(c\beta+\sum_{(i,j)\in\mathbb{I}}c_{ij}[S]^{-i}[T]^{-j})=\chi(c\beta)+\sum_{(i,j)\in\mathbb{I}}\chi(c_{ij}[S]^{-i}[T]^{-j}),
\]
so an element of $H$ is determined entirely by its values on the elements $c\beta$ and $c_{ij}[S]^{-i}[T]^{-j}$, and $$H\cong \Hom(W(\Fp), W(\Fp))\times\prod_{\substack{(i,j)\in\mathbb{I}}}\Hom(W(k), W(\Fp)).$$
But since $W(k)\cong\Hom(W(k), W(\Fp))$ by the well known trace map isomorphism (see Remark 4.10 in \cite{KostersWan}), 
	\[
		H\cong  W(\Fp)\times\prod_{\substack{(i,j)\in\mathbb{I}}}W(k).
	\]
As the sets $\mathbb{I}$ and $\mathbb{J}_S\cup\mathbb{J}_T$  are bijective,
\[
	\prod_{\substack{(i,j)\in\mathbb{I}}}W(k)= \prod_{\substack{(i,j)\in\mathbb{J}_S}}W(k)\prod_{\substack{(i,j)\in\mathbb{J}_T}}W(k), 
\]
and the lemma follows.

\end{proof}

Hence Parshin's reciprocity map in (\ref{line:parshinrec}) induces an isomorphism
\begin{align}\label{VK_iso}
\phi: \widehat{K_2^{top}(K)}\xrightarrow{\sim} W(\Fp)\times \prod_{\substack{(i,j)\in\mathbb{J}_S}}W(k)\prod_{\substack{(i,j)\in\mathbb{J}_T}}W(k).
\end{align}

Using Corollary~\ref{explicit_symbol}, 
we will explicitly compute $\phi$.
For $X=S$ or $T$, define two subsets of $VK_2^{top}(K)$:
\begin{align*}
	VK_{2,X}^{top}(K) &= \left\{\prod_{\substack{(i,j)\in\mathbb{J}_X\\k\in\mathbb{Z}_{\geq 0}}}  \{1+a_{ijk}S^iT^j, X\}^{c_{ijk}p^k}\ |\ a_{ijk}\in k, c_{ijk}\in[0,p-1]\right\},%\\
	\end{align*}
so that by Proposition~\ref{VK_decomp}, $VK_{2}^{top}(K)\cong VK_{2,S}^{top}(K)\times VK_{2,T}^{top}(K)$.
Define the map
$\phi_S:VK_{2,S}^{top}(K)\to \prod_{(i,j)\in\mathbb{J}_S}W(k)$
by sending $\prod_{\substack{(i,j)\in\mathbb{J}_S,\\ k\in\mathbb{Z}_{\geq 0}}} \{1+a_{ijk}S^iT^j, S\}^{c_{ijk}p^k}$ to 	
\[
	\left (\sum_{k=0}^\infty c_{ijk}p^k\sum_{\substack{(i,j)\in\mathbb{J}_S\\\frac{m}{i}=\frac{n}{j}\in\mathbb{Z}_{\geq 0}}}(-j)[-a_{ijk}]^{m/i}\right )_{(m,n)\in\mathbb{J}_S},
\]
and define $\phi_T: VK_{2,T}^{top}(K)\to \prod_{(i,j)\in\mathbb{J}_T}W(k)$ by mapping 
$\prod_{\substack{(i,j)\in\mathbb{J}_T,\\ k\in\mathbb{Z}_{\geq 0}}} \{1+b_{ijk}S^iT^j, T\}^{d_{ijk}p^k}$ to
\[
 \left (\sum_{k=0}^\infty d_{ijk}p^k\sum_{\substack{(i,j)\in\mathbb{J}_T\\\frac{m}{i}=\frac{n}{j}\in\mathbb{Z}_{\geq 0}}}i[-b_{ijk}]^{m/i}\right )_{\substack{(m,n)\in\mathbb{J}_T}}.
 \]

\begin{proposition}\label{prop:phi}
Parshin's reciprocity map induces an isomorphism:
\begin{align*}
	\phi: \widehat{K_2^{top}(K)} &\to W(\Fp)\times \prod_{\substack{(i,j)\in\mathbb{J}_S}}W(k)\prod_{\substack{(i,j)\in\mathbb{J}_T}}W(k).
\end{align*}
Writing each $y\in \widehat{K_2^{top}(K)}$ as in Corollary~\ref{K2top_structure}, the map $\phi$ is given by sending
\begin{align*}
	y=\{S^e, T\}\prod\{1+a_{ij}S^iT^j, S\}\cdot\prod\{1+b_{ij}S^iT^j,T\} 
\end{align*}
to the tuple
\begin{align*}
  \left (\Tr_{W(k)/W(\Fp)}(e\beta), \phi_S(\prod\{1+a_{ij}S^iT^j, S\}), \phi_T(\{1+b_{ij}S^iT^j,T\} )\right ).
\end{align*}
\end{proposition}
\begin{proof}
The proposition follows by starting with the map (\ref{line:parshinrec}) and Corollary~\ref{explicit_symbol}, and following the isomorphisms in Lemma~\ref{lemma:Hiso}. As an example, we will show how this process works for the $\{S^e,T\}$ term. The other two components are similar.

Write $\chi(x)=(\omega_K^H(\{S^e, T\}))(x)= [x,\{S^e, T\})\in H$. If we write $x=c\beta+\sum_{(i,j)\in\mathbb{I}}c_{ij}[S]^{-i}[T]^{-j}$ as in Corollary~\ref{structurethm}, we see  by Corollary~\ref{explicit_symbol} that $\chi(x)=\chi(c\beta)$. We can therefore just consider the case when $x=c\beta$ and so:
\begin{align}\label{line:tracecb}
\chi(c\beta) = \Tr_{W(k)/W(\Fp)}(c\beta e) = c\Tr_{W(k)/W(\Fp)}(\beta e).
\end{align}
We can consider $\chi$ to be an element in $\Hom(W(\Fp), W(\Fp))$, mapping $c$ to $c\Tr_{W(k)/W(\Fp)}(\beta e)$. But then recall the isomorphism:
\begin{align*}
	\Hom(W(\Fp), W(\Fp)) &\to W(\Fp)\\
	(x\mapsto xa)\mapsto a,
\end{align*}
and so under this map, $\chi$ corresponds to $\Tr_{W(k)/W(\Fp)}(\beta e)$ by (\ref{line:tracecb}).  Overall, we see that 
\[
	\phi(\{S^e, T\}) = \left (\Tr_{W(k)/W(\Fp)}(\beta e), 0, 0 \right ).
\]
\end{proof}

\begin{example}
To see how the above maps compare with the map in \cite{KostersWan} (Remark 4.10), we shall compute a couple of examples. 
Observe that,
	\[
		\{1+a_{1,1,0}S^1T^1, S\}\xrightarrow{\phi_S}
		\left (\begin{cases} 0\textrm{ if } m\neq n\\
					[a_{1,1,0}]^n \textrm{ if } m=n
			\end{cases}
		\right )_{(m,n)\in\mathbb{J}_S}
	\]
	\[
		\{1+a_{1,p,0}S^1T^p, T\}\xrightarrow{\phi_T}
		\left (\begin{cases} 
				[-a_{1,p,0}]^m \textrm{ if } n=pm \\
				0\textrm{ otherwise}
			\end{cases}
		\right )_{(m,n)\in\mathbb{J}_T}.
	\]
To compare, the map in \cite{KostersWan} (and more generally Lemma 2.1 in \cite{Katz}) yields:
	\[
		(1-a_{10}T^1)^{p^0}\mapsto ([a_{10}]^i)_{\gcd(i,p)=1}.
	\]
\end{example}

\ramgroups*
\begin{proof}
The proof will follow from computing $\phi(U^{\vec{r}}K_2^{top}(K))$ via Proposition~\ref{prop:phi}.

Let $y\in VK_2^{top}(K)$, and write $y$ as in Proposition~\ref{VK_decomp}. Then $y\in U^{\vec{r}}K_2^{top}(K)$ if and only if for all $(p^ki, p^kj)<\vec{r}$, $a_{ijk}=0$ and for all $(p^ki, p^kj)<\vec{r}$, $b_{ijk}=0$. Hence 
\[
	\phi(U^{\vec{r}}K_2^{top}(K))\subset \left (p^{\ell(\vec{r}, \vec{m})}W(k)\right)_{\substack{\vec{m}\in\mathbb{Z}_{>0}^2\\ p\nmid \gcd\vec{m}}}.
\]

The other inclusion is clear from the explicit definition of the maps $\phi_S$ and $\phi_T$. 
\end{proof}

%%%%%%%%%%%%%%%%%%%%%%%%%%%%%%%%%%%%%%%%%%%%%%
% Bibliography
%%%%%%%%%%%%%%%%%%%%%%%%%%%%%%%%%%%%%%%%%%%%%%

\end{document}